\documentclass[a4paper,11pt]{article}

\usepackage{xr}
\usepackage{amssymb,amsmath,amsthm}
\usepackage{mathtools}
\usepackage{xspace}
\usepackage{multicol}
\usepackage{mathrsfs}
\usepackage{bigstrut}
\usepackage{bm}
\theoremstyle{plain}
\newtheorem{theorem}{Theorem}
\newtheorem{lemma}[theorem]{Lemma}
\newtheorem{proposition}[theorem]{Proposition}
\theoremstyle{definition}
\newtheorem{definition}[theorem]{Definition}
\numberwithin{theorem}{subsection}
\newenvironment{smallpmatrix}{\left(\begin{smallmatrix}}{\end{smallmatrix}\right)} 
\newcommand{\eqdef}{=}
\DeclarePairedDelimiter{\brax}{(}{)}
\DeclarePairedDelimiter{\sqbrax}{[}{]}
\DeclarePairedDelimiter{\setbrax}{\{}{\}}
\DeclarePairedDelimiter{\set}{\{}{\}}
\DeclarePairedDelimiter{\abs}{\lvert}{\rvert}

\DeclarePairedDelimiter{\normDoubleBar}{\lVert}{\rVert}
\newcommand{\open}[2]{\brax{#1 , #2}}
\newcommand{\clsd}[2]{\sqbrax{#1 , #2}}
\newcommand{\suchthat}{:}
\renewcommand{\vec}[1]{\bm{#1}}
\newcommand{\vecsuper}[2]{\vec{#1}^{(#2)}}
\newcommand{\tbbQ}[0]{$\mathbb{Q}$\xspace}
\newcommand{\bbC}[0]{\mathbb{C}}
\newcommand{\bbF}[0]{\mathbb{F}}
\newcommand{\bbN}[0]{\mathbb{N}}
\newcommand{\bbP}[0]{\mathbb{P}}
\newcommand{\bbQ}[0]{\mathbb{Q}}
\newcommand{\bbR}[0]{\mathbb{R}}
\newcommand{\bbZ}[0]{\mathbb{Z}}
\newcommand{\frakS}[0]{\mathfrak{S}}
\newcommand{\frakI}[0]{\mathfrak{I}}
\newcommand{\singSeriesOf}[1]{\frakS_{#1}}
\newcommand{\singIntegralOf}[1]{\frakI_{#1}}
\newcommand{\weightbox}{\mathscr{B}}
\newcommand{\cancellation}{{\mathscr C}}
\newcommand{\rank}{\operatorname{rank}}
\newcommand{\sing}{\operatorname{Sing}}
\newcommand{\supnorm}[1]{\normDoubleBar{#1}_\infty}
\newcommand{\supnormbigg}[1]{\normDoubleBar[\bigg]{#1}_\infty}
\newcommand{\aDotCapitalF}{\vec{a}\cdot\vec{F}}
\newcommand{\betaDotLittleF}{\vec{\beta}\cdot\vec{f}}
\newcommand{\leadingPart}{\vec{f}^{[d]}}
\newcommand{\gradFMultilinear}[2]{\vec{m}^{(#1 \cdot \vec{f})} \brax{ #2 }}
\newcommand{\gradSomethingMultilinear}[2]{\vec{m}^{(#1)} \brax{ #2 }}
\newcommand{\fMultilinJacobian}[1]{J_{\vec{m}}^{\brax{f}}\brax{#1}}
\newcommand{\betaFMultilinJacobian}[1]{J_{\vec{m}}^{\brax{\vec{\beta}\cdot\vec{f}}}\brax{#1}}
\newcommand{\somethingDotFMultilinJacobian}[2]{J_{\vec{m}}^{\brax{#1\cdot\vec{f}}}\brax{#2}}
\newcommand{\HMultilinJacobian}[1]{J_{\vec{m}}^{\brax{H}}\brax{#1}}
\newcommand{\betaDotHMultilinJacobian}[1]{J_{\vec{m}}^{\brax{\vec{\beta}\cdot\vec{H}}}\brax{#1}}
\newcommand{\numZeroesInBoxOf}[1]{N_{#1,\weightbox}}
\newcommand{\auxIneqOfSomethingNumSolns}[1]{N^{\operatorname{aux}}_{#1}}
\newcommand{\auxIneqNumSolns}{N^{\operatorname{aux}}_{\vec{\beta}\cdot\vec{f}}}
\newcommand{\worstTangentSpaceDimCapitalH}{\sigma^\ast(\vec{H})}
\newcommand{\worstTangentSpaceDimCapitalF}{\sigma^\ast(\vec{F})}
\newcommand{\worstTangentSpaceDimLeadingPart}{\sigma^\ast(\leadingPart)}

\begin{document}

				\title{Systems of forms in many variables}
				\author{S. L. Rydin Myerson}
				\maketitle

\begin{abstract}
	We consider systems $\vec{F}(\vec{x})$ of $R$ homogeneous forms  of the same degree $d$ in $n$ variables with integral coefficients. If $n\geq d2^dR+R$ and the coefficients of $\vec{F}$ lie in an explicit Zariski open set, we give  a nonsingular Hasse principle for the equation $\vec{F}(\vec{x})=\vec{0}$, together with an asymptotic formula for the number of solutions to in integers of bounded height. This improves on the number of variables needed in previous results for general systems $\vec{F}$ as soon as the number of equations $R$ is at least 2 and the degree $d$ is at least 4.
\end{abstract}

\section{Introduction}

\subsection{Results}

Let $\vec{F}(\vec{x}) = \brax{F_1(\vec{x}),\dotsc,F_R(\vec{x})}^T$ be a vector of $R$ homogeneous forms of the same degree $d$, where $d\geq 2$, in $n$ variables  $\vec{x} =(x_1,\dotsc,x_n)^T$ and having integral coefficients. We write $V(\vec{F})$ for the projective variety in $\bbP_\bbQ^{n-1}$ cut out by the condition $\vec{F}(\vec{x})=\vec{0}$. Our main result, proved at the end of \S\ref{3.sec:aux_ineq}, is as follows:

\begin{theorem}\label{3.thm:main_thm_short}
		Let $\weightbox$ be a box in $\bbR^n$, contained in the box $\clsd{-1}{1}^n$ and having sides of length at most 1 which are parallel to the coordinate axes. For each $ P\geq 1$, write
		\begin{equation*}
		\numZeroesInBoxOf{\vec{F}}(P)
		=
		\#
		\set{ \vec{x}\in\bbZ^n \suchthat
			\vec{x}/P\in\weightbox,\,
			\vec{F}(\vec{x})=\vec{0}
		}.
		\end{equation*}
		Suppose that $\vec{F}\in U_{d,n,R}(\bbQ)$ for some explicit, nonempty Zariski open set $U_{d,n,R}$ which will be defined in Proposition~\ref{3.prop:general_position_condition} below. If we have
		\begin{equation}
		n
		>
		d2^d R+R.
		\label{3.eqn:condition_on_n_short}
		\end{equation}
		then for all $P\geq 1$ we have
		\begin{equation*}
		\numZeroesInBoxOf{\vec{F}}(P)
		=
		\singIntegralOf{\vec{F},\weightbox}\singSeriesOf{\vec{F}} P^{n-dR}
		+
		O\brax{P^{n-dR-\delta}}.
		\end{equation*}
		Here the implicit constant and the constant $\singSeriesOf{\vec{F}}$ depend only on the forms $F_i$, the constant  $\singIntegralOf{\vec{F},\weightbox}$ depends only on $\vec{F}$ and $\weightbox$, and $\delta$ is a positive constant depending only on $d$ and $R$. If  $V(\vec{F})$ has dimension $n-1-R$ and a smooth real point whose homogenous co-ordinates lie in $\weightbox$, then $\singIntegralOf{\vec{F},\weightbox}$ is positive. If $V(\vec{F})$ has dimension $n-1-R$ and a smooth point over $\bbQ_p$ for each prime $p$, then $\singSeriesOf{\vec{F}} $ is positive. 
	\end{theorem}

The requirement that $\vec{F}\in U_{d,n,R}(\bbQ)$ is satisified for 100\% of systems $\vec{F}$, and can in principle be tested algorithmically for any particular system $\vec{F}$ with integral coefficients. In future work we will remove this hypothesis at the cost of an increased number of variables.

Roughly speaking, to have $\vec{F}\in U_{d,n,R}(\bbQ)$ means that all the tangent spaces to some auxiliary varieties should have codimension $n-R+1$ or greater. In a sense then we ask that these auxiliary varieties should not be too singular; see \S\ref{3.sec:aux_ineq} for more details. This does not appear to have a natural interpretation in terms of the original equations $\vec{F}(\vec{x})=\vec{0}$.

When $d =2$ or 3, previous work of the author provides the same conclusion with the condition $\vec{F}\in U_{d,n,R}(\bbQ)$ replaced by the condition that $V(\vec{F})$ be smooth of dimension $n-R-1$. See Theorem~\ref{1.thm:main_thm_short} and the comments after Lemma~\ref{1.lem:nonsing_case} in~\cite{quadSystemsManyVars} for the case $d=2$ and  see~Theorem~1.2 in~\cite{systemsCubicForms} for the case $d=3$. The case of interest in the theorem is thus $d\geq 4$.

We outline what is known in that case. A longstanding result of Birch~\cite{birchFormsManyVars} gives the conclusion of Theorem~\ref{3.thm:main_thm_short} with the conditions $\vec{F}\in U_{d,n,R}(\bbQ)$ and  \eqref{3.eqn:condition_on_n_short} replaced by
\begin{equation}\label{3.eqn:birch's_condition_long}
n-1-\dim W
> (d-1)2^{d-1}R(R+1),
\end{equation}
where $W$ is the projective variety cut out in $\bbP_\bbQ^{n-1}$ by the condition that the $R\times n$ Jacobian matrix $\brax{\partial F_i(\vec{x})/ \partial x_j}_{ij}$ has rank less than $R$. In particular, if $V(\vec{F})$ is smooth of dimension $n-R-1$ then  \eqref{3.eqn:birch's_condition_long} holds whenever
\begin{equation*}
n \geq
(d-1)2^{d-1}R(R+1)+R,
\end{equation*}
see for example Lemma~1.1 in~\cite{quadSystemsManyVars}. There is a refinement of \eqref{3.eqn:birch's_condition_long} due to Dietmann~\cite{dietmannWeylsIneq} and to Schindler~\cite{schindlerWeylsIneq}. They show that it suffices to have
\begin{equation}\label{3.eqn:d-s_condition}
n-\sigma_\bbZ(\vec{F})>(d-1)2^{d-1}R(R+1),
\end{equation}
where we write
\begin{equation*}
\sigma_\bbZ(\vec{F})= 1+\max_{\vec{a}\in\bbZ^R\setminus\set{\vec{0}}} \dim\sing (\aDotCapitalF).
\end{equation*}
When $R=1$ this is identical to Birch's condition, but once $R\geq 2$ it may be weaker.

Provided that $d\geq 4$ and $R\geq 2$ our condition \eqref{3.eqn:condition_on_n_short} on the number of variables is weaker than any of \eqref{3.eqn:birch's_condition_long}--\eqref{3.eqn:d-s_condition}, since
\[
d2^dR+R
<
3\cdot(d-1)2^{d-1}R
\leq 
(d-1)2^{d-1}R(R+1).
\]
For example, when $d=4$ and $R=2$, Birch's result applies to smooth intersections of pairs of quartics in $n\geq 148$ variables, while \eqref{3.eqn:condition_on_n_short} requires $n\geq 130$ for pairs of quartics in general position.

In the case when $R=1$ stronger results are available. For a smooth quartic hypersurface Hanselmann~\cite{hanselmannQuartics40Vars} gives the condition $n\geq 40$ in place of the $n \geq 49$ required to apply Birch's result, and work in progress of Marmon and Vishe yields a further improvement. When $R=1$ and $d\geq 5$, a sharper condition than \eqref{1.eqn:birch's_condition_long} is available by work of Browning and Prendiville~\cite{browningPrendivilleImprovements}. For a smooth hypersurface with $5\leq d\leq 9$ this is essentially a reduction of 25\% in the number of variables required.

\subsection{The auxiliary inequality}\label{3.sec:aux_ineq}

By previous work of the author~\cite[Theorem~\ref{1.thm:manin}]{quadSystemsManyVars} it will suffice to bound the number of solutions to a certain multilinear inequality. We quote the following definition from~\cite[Definition~\ref{1.def:aux_ineq}]{quadSystemsManyVars}.

\begin{definition}\label{3.def:aux_ineq}
	For each $k \in\bbN\setminus\set{\vec{0}}$ and $\vec{t}\in\bbR^k$ we write  $\supnorm{\vec{t}} = \max_i \abs{t_i}$ for the supremum norm. Let $f(\vec{x})$ be any polynomial of degree $d\geq 2$ with real coefficients in $n$ variables $x_1,\dotsc,x_n$.  For $i= 1,\dotsc, n$ we define
	\begin{equation*}
	m^{( f )}_i \brax{\vec{x}^{(1)},\dotsc,\vec{x}^{(d-1)} }
	=
	\sum_{j_1,\dotsc,j_{d-1}=1}^n
	x^{(1)}_{j_1} \dotsm x^{(d-1)}_{j_{d-1}}
	\frac{\partial^{d} f(\vec{x})}{\partial x_{j_1} \dotsm \partial x_{j_{d-1}} \partial x_i},
	\end{equation*}
	where we write $\vecsuper{x}{j}$  for a vector of $n$ variables $(x^{(j)}_1,\dotsc,x^{(j)}_n)^T$. This defines an $n$-tuple of multilinear forms
	\[
	\gradSomethingMultilinear{ f } {\vec{x}^{(1)},\dotsc,\vec{x}^{(d-1)} }\in \bbR[\vec{x}^{(1)},\dotsc,\vec{x}^{(d-1)}]^n.
	\]
	Finally, for each  $B \geq 1$  we put $\auxIneqOfSomethingNumSolns{f} \brax{ B }$ for the number of $(d-1)$-tuples of integer $n$-vectors $\vec{x}^{(1)}, \dotsc, \vec{x}^{(d-1)}$ with
	\begin{gather}
	\supnorm{\vecsuper{x}{1}},\dotsc,\supnorm{\vecsuper{x}{d-1}} \leq B, 
	\nonumber
	\\
	\label{3.eqn:aux_ineq}
	\supnorm{\gradSomethingMultilinear{ f }{ \vec{x}^{(1)}, \dotsc, \vec{x}^{(d-1)} }} < \supnorm{ f^{[d]} } B^{d-2}
	\end{gather}
	where we let $\supnorm{f^{[d]}} = \frac{1}{d!} \max_{\vec{j}\in\set{1,\dotsc,n}^d} \abs[\big]{\frac{\partial^d f(\vec{x})}{\partial x_{j_1}\dotsm\partial x_{j_d}}}$.
\end{definition}

Our results will involve a quantity $\worstTangentSpaceDimCapitalH$ defined as follows.

\begin{definition}\label{3.def:multilin_Jacobian}
	Suppose that $f(\vec{x})$ is a polynomial of degree $d$ in $n$ variables, and that $d\geq 2$, and let $\gradSomethingMultilinear{f}{\vecsuper{x}{1},\dotsc,\vecsuper{x}{d-1}}$ be as in Definition~\ref{3.def:aux_ineq}. Then we write $\fMultilinJacobian{\vecsuper{x}{1},\dotsc,\vecsuper{x}{d-1}}$ for the $n\times (d-1)n$ Jacobian matrix of $\gradSomethingMultilinear{f}{\vecsuper{x}{1},\dotsc,\vecsuper{x}{d-1}}$, that is,
	\begin{multline}\label{3.eqn:def_of_multilin_Jacobian}
	\fMultilinJacobian{\vecsuper{x}{1},\dotsc,\vecsuper{x}{d-1}}
	=
	\\
	\left(
	\begin{array}{@{}c|c|c|c@{}}
	\frac{\partial\gradSomethingMultilinear{f}{\vecsuper{x}{1},\dotsc,\vecsuper{x}{d-1}}}{\partial x^{(1)}_1}
	&
	\frac{\partial\gradSomethingMultilinear{f}{\vecsuper{x}{1},\dotsc,\vecsuper{x}{d-1}}}{\partial x^{(1)}_2}
	&
	\cdots
	&
	\frac{\partial\gradSomethingMultilinear{f}{\vecsuper{x}{1},\dotsc,\vecsuper{x}{d-1}}}{\partial x^{(d-1)}_n}
	\end{array}
	\right).
	\end{multline}
	If $\vec{H}(\vec{x})$ is a system of $R$ homogeneous polynomials of the same degree $d$ in $n$ variables, with coefficients in a field $\bbF$, then we set
	\begin{equation}\label{3.eqn:def_of_sigma-star}
	\worstTangentSpaceDimCapitalH
	\eqdef
	n-
	\min_{\vec{\beta}\in\bar{\bbF}^R\setminus\set{\vec{0}} }
	\min_{\substack{ \vecsuper{x}{1}, \dotsc,\vecsuper{x}{d-1}\in\bar{\bbF}^n\setminus\set{\vec{0}} \\ \gradSomethingMultilinear{\vec{\beta}\cdot\vec{H}}{\vecsuper{x}{1},\dotsc,\vecsuper{x}{d-1}}=\vec{0} }} 
	\rank \betaDotHMultilinJacobian{\vecsuper{x}{1},\dotsc,\vecsuper{x}{d-1}}
	\end{equation}
	where $\bar{\bbF}$ is an algebraic closure of $\bbF$.
\end{definition}

In \eqref{3.eqn:def_of_sigma-star} one could think of $\rank \betaDotHMultilinJacobian{\vecsuper{x}{1},\dotsc,\vecsuper{x}{d-1}}$ as the codimension of the tangent space to the variety $\gradSomethingMultilinear{\vec{\beta}\cdot\vec{H}}{\vecsuper{x}{1},\dotsc,\vecsuper{x}{d-1}}=\vec{0}$ at the point $\brax{\vecsuper{x}{1},\dotsc,\vecsuper{x}{d-1}}$. In this sense $\worstTangentSpaceDimCapitalH$ could be said to measure the extent to which these varieties are singular. It does not however seem to be fruitful to pursue this interpretation further.

Our upper bound for $\auxIneqOfSomethingNumSolns{f} \brax{ B }$ in terms of this quantity $\worstTangentSpaceDimLeadingPart$ is as follows. The proof is in \S\ref{3.sec:aux_ineq_bound_proof}.

\begin{proposition}\label{3.prop:aux_ineq_bound_general_pos}
	Let  $\gradSomethingMultilinear{f}{\vecsuper{x}{1},\dotsc,\vecsuper{x}{d-1}}$ be as in Definition~\ref{3.def:aux_ineq} and let $\worstTangentSpaceDimCapitalH$ be as in Definition~\ref{3.def:multilin_Jacobian}. For all $\vec{\beta}\in\bbR^R$ and all $B \geq 1$ we have
	\begin{equation*}
	\auxIneqNumSolns\brax{ B }
	\ll_{\vec{f}}
	B^{(d-2)n+\worstTangentSpaceDimLeadingPart}
	\brax{\log B}^{d-1}.
	\end{equation*}
\end{proposition}

The following result, proved in \S\ref{3.sec:general_pos_aux_ineq_bound}, shows  that $\worstTangentSpaceDimCapitalH$ is typically quite small.

\begin{proposition}\label{3.prop:general_position_condition}
		Let $\worstTangentSpaceDimCapitalH$ be as in Definition~\ref{3.def:multilin_Jacobian} Suppose that $n\geq R$ holds. We may consider the space of $R$-tuples of homogeneous degree $d$ forms in $n$ variables as an affine space defined over \tbbQ. The condition that  $\worstTangentSpaceDimCapitalH\leq R-1$ defines a nonempty Zariski open set $U_{d,n,R}$ in this space.
\end{proposition}

We  deduce Theorem~\ref{3.thm:main_thm_short}.

\begin{proof}[Proof of Theorem~\ref{3.thm:main_thm_short}]
	Note that \eqref{3.eqn:condition_on_n_short} certainly implies $n\geq R$ and so $U_{d,n,R}$ is a nonempty Zariski open set, by Proposition~\ref{3.prop:general_position_condition}. The condition $\vec{F}\in U_{d,n,R}(\bbQ)$ means exactly that $\worstTangentSpaceDimCapitalF\leq R-1$, and so by Proposition~\ref{3.prop:aux_ineq_bound_general_pos} we hve
	\begin{align*}
	\auxIneqOfSomethingNumSolns{\vec{\beta}\cdot\vec{f}}\brax{B}
	&\ll_{\vec{f},\epsilon}
	B^{(d-2)n+R-1+\epsilon}
	\\
	&\ll_{\vec{f}}
	B^{(d-1)n-2^d \cancellation}
	\end{align*}
	where
	\[
	\cancellation = \frac{n-R+\tfrac{1}{2}}{2^d}.
	\]
	We have $\cancellation > dR$, by our assumption \eqref{3.eqn:condition_on_n_short}. The conclusion of the theorem now follow on applying Theorem~\ref{1.thm:manin} from~\cite{quadSystemsManyVars}.
\end{proof}

\section{Counting solutions to the auxiliary inequality}\label{3.sec:general_pos_aux_ineq_bound}

In this section we prove Proposition~\ref{3.prop:aux_ineq_bound_general_pos}. We begin with a   lemma giving an analytic intepretation of the quantity $\sigma^\ast$ from Definition~\ref{3.def:multilin_Jacobian}.

\subsection{Finding spaces on which the Jacobian is large}\label{3.sec:aux_ineq_quantitative_nonsing}

The  result below shows that, provided $\worstTangentSpaceDimLeadingPart$ is small, then for every point where $\supnorm{\gradFMultilinear{\vec{\beta}}{\vecsuper{x}{1},\dotsc,\vecsuper{x}{d-1}}}$ is small, there are many ways in which we can perturb the variables $\vecsuper{x}{i}$ such that $\supnorm{\vecsuper{m}{\betaDotLittleF}}$ increases rapidly.

\begin{lemma}\label{3.lem:Jacobian_of_grad_f}
	Let $\fMultilinJacobian{\vecsuper{x}{1},\dotsc,\vecsuper{x}{d-1}}$ and $\sigma^\ast$  be as in Definition~\ref{3.def:multilin_Jacobian}. Suppose that $\vec{\beta}\in\bbR^R\setminus\set{\vec{0}}$ and that $\vecsuper{x}{1},\dotsc,\vecsuper{x}{d-1}\in \bbR^n\setminus\set{\vec{0}}$. Then one of the following two alternatives holds: either we have
	\begin{equation*}
	\supnorm{\gradFMultilinear{\vec{\beta}}{\vecsuper{x}{1},\dotsc,\vecsuper{x}{d-1}}}
	\gg_{\vec{f}}
	\supnorm{\vec{\beta}} \supnorm{\vecsuper{x}{1}}\dotsm\supnorm{\vecsuper{x}{d-1}},
	\end{equation*}
	or else there exist linear subspaces $U_1,\dotsc,U_{d-1}$ of $ \bbR^n$, satisfying
	\[
	\dim U_1 +\dotsb+\dim U_{d-1}=n-\worstTangentSpaceDimLeadingPart,
	\]
	such that for all $\vecsuper{u}{1}\in U_1, \dotsc, \vecsuper{u}{d-1}\in U_{d-1}$, we have
	\begin{multline*}
	\supnormbigg{ \betaFMultilinJacobian{\vecsuper{x}{1},\dotsc,\vecsuper{x}{d-1}}  
		\begin{smallpmatrix}\vecsuper{u}{1} \bigstrut \\ \hline \raisebox{5pt}{\scalebox{.75}{\vdots}} \\ \hline \vecsuper{u}{d-1} \end{smallpmatrix}
	} 
	\\
	\gg_{\vec{f}}
	\supnorm{\vec{\beta}}
	\supnorm{\vecsuper{x}{1}}\dotsm\supnorm{\vecsuper{x}{d-1}}
	\end{multline*}
	Furthermore we may take the spaces $U_i$ to be  spanned by standard basis vectors of $\bbR^n$.
\end{lemma}

We give a proof after stating the following simple lemma on real matrices, which is Lemma~3.2(iii) in~\cite{systemsCubicForms}.

\begin{lemma}\label{3.lem:space_where_matrix_is_small}
	Let $M$ be a real $m\times n$ matrix. Let $k\in\bbN$ such that $k \leq\min\setbrax{ m,n }$ holds. For any $C\geq 1,$ either there is an $(n-k+1)$-dimensional linear subspace $X$ of $\bbR^n$ such that
	\begin{align*}
	\supnorm{M \vec{X}}
	&\leq
	C^{-1}\supnorm{\vec{X}}
	&\text{for all }\vec{X}\in X,
	\end{align*}
	or there is a $k$-dimensional linear subspace $V $ of $\bbR^n,$ spanned by standard basis vectors of $\bbR^n,$ such that
	\begin{align*}
	\supnorm{M \vec{v}}
	&\gg_{m,n}
	C^{-1}\supnorm{\vec{v}}
	&\text{for all }\vec{v}\in V.
	\end{align*}
\end{lemma}

\begin{proof}[Proof of Lemma~\ref{3.lem:Jacobian_of_grad_f}]
	Suppose that $\vec{\gamma}\in \bbR^R$ and $\vecsuper{z}{1},\dotsc,\vecsuper{z}{d-1}\in \bbR^n$ such that 
	\[
	\supnorm{\vec{\gamma}} =\supnorm{\vecsuper{z}{1}}=\dotsb=\supnorm{\vecsuper{z}{d-1}}= 1
	\]
	holds. We will show that either
	\begin{align}\nonumber
	\supnorm{\gradFMultilinear{\vec{\gamma}}{\vecsuper{z}{1},\dotsc,\vecsuper{z}{d-1}}}
	&\gg_{\vec{f}} 1,
	\intertext{or else}
	\label{3.eqn:many_normals_alternate_formulation}
	\supnorm{ \somethingDotFMultilinJacobian{\vec{\gamma}}{\vecsuper{z}{1},\dotsc,\vecsuper{z}{d-1}}  
		\vec{u}
	}
	&\gg_{\vec{f}}
	\supnorm{\vec{u}}
	&\text{for all }\vec{u}\in U,
	\end{align}
	for some $(n-\worstTangentSpaceDimLeadingPart)$-dimensional linear subspace $U$ of $\bbR^{(d-1)n}$ spanned by standard basis vectors of $(d-1)n$-dimensional space. Once we have shown this, the result will follow on writing
	\begin{align*}
	\vec{\gamma} &= \vec{\beta}/\supnorm{\vec{\beta}},
	&
	\vecsuper{z}{i} &= \vecsuper{x}{i} / \supnorm{\vecsuper{x}{i}},
	\\
	U &= U_1 \times\dotsb\times U_{d-1}
	&
	\vec{u}
	&=\left(
	\begin{array}{@{}c@{}}\vecsuper{u}{1}/\supnorm{\vecsuper{x}{1}}
	\bigstrut \\ \hline \raisebox{2pt}{\vdots} \\ \hline \vecsuper{u}{d-1}/\supnorm{\vecsuper{x}{d-1}}\bigstrut
	\end{array}\right).
	\end{align*}
	
	Let $C\geq 1$ and apply Lemma~\ref{3.lem:space_where_matrix_is_small} with the choices $k=n-\worstTangentSpaceDimLeadingPart$ and $M=\somethingDotFMultilinJacobian{\vec{\gamma}}{\vecsuper{z}{1},\dotsc,\vecsuper{z}{d-1}}  $. This shows that either
	\begin{align*}
	\supnorm{ \somethingDotFMultilinJacobian{\vec{\gamma}}{\vecsuper{z}{1},\dotsc,\vecsuper{z}{d-1}}  
		\vec{u}
	}
	&\geq C^{-1}
	\supnorm{\vec{u}}
	&\text{for all }\vec{u}\in U,
	\end{align*}
	for some  $(n-\worstTangentSpaceDimLeadingPart)$-dimensional linear subspace $U$ of $\bbR^{(d-1)n}$ spanned by standard basis vectors, or else there is a $(1+\worstTangentSpaceDimLeadingPart)$-dimensional linear subspace $X$ of $\bbR^{(d-1)n}$ such that
	\begin{align}\label{3.eqn:multlin_jacobian_small}
	\supnorm{ \somethingDotFMultilinJacobian{\vec{\gamma}}{\vecsuper{z}{1},\dotsc,\vecsuper{z}{d-1}}   \vec{X}}
	&\leq
	C^{-1}\supnorm{\vec{X}}
	&\text{for all }\vec{X}\in X.
	\end{align}
	
	Suppose for a contradiction that \eqref{3.eqn:many_normals_alternate_formulation} is false for every $U$ satisfying the required conditions. Then for each $C\geq 1$ there exist vectors $\vec{\gamma},\vecsuper{z}{1},\dotsc,\vecsuper{z}{d-1}$ with unit norm, and a space $X$ with dimension $1+\worstTangentSpaceDimLeadingPart$, satisfying \eqref{3.eqn:multlin_jacobian_small}. Passing to a convergent subsequence, we find vectors $\vec{\gamma},\vecsuper{z}{1},\dotsc,\vecsuper{z}{d-1}$ with unit norm and a space $X$ with dimension $(d-2)n+1+\worstTangentSpaceDimLeadingPart$, such that 
	\begin{align*}
	\somethingDotFMultilinJacobian{\vec{\gamma}}{\vecsuper{z}{1},\dotsc,\vecsuper{z}{d-1}}   \vec{X}
	&=\vec{0}
	&\text{for all }\vec{X}\in X.
	\end{align*}
	In other words, the matrix $\somethingDotFMultilinJacobian{\vec{\gamma}}{\vecsuper{z}{1},\dotsc,\vecsuper{z}{d-1}}$ has rank $n-\worstTangentSpaceDimLeadingPart-1$ or less. But this is impossible, by the definition \eqref{3.eqn:def_of_sigma-star} of the quantity $\sigma^\ast$. This proves the result.
\end{proof}

\subsection{Proof of Proposition~\ref{3.prop:aux_ineq_bound_general_pos}}\label{3.sec:aux_ineq_bound_proof}

We use Lemma~\ref{3.lem:Jacobian_of_grad_f} to bound the counting function $\auxIneqNumSolns(B)$ by covering the set of solutions to the auxiliary inequality \eqref{3.eqn:aux_ineq} with a collections of boxes of controlled size.

\begin{proof}[Proof of Proposition~\ref{3.prop:aux_ineq_bound_general_pos}]
	If $\vec{\beta}=\vec{0}$ then $\auxIneqNumSolns(B)=0$ and the result is trivial. Let $\vec{\beta}\in\bbR^R\setminus\set{\vec{0}}$. For each $T_1,\dotsc,T_{d-1} \geq 1$, define
	\begin{multline*}
	Z(T_1,\dotsc,T_{d-1})
	\\
	\eqdef
	\Big\{
	\brax{\vecsuper{x}{1},\dotsc,\vecsuper{x}{d-1}}\in\brax{\bbZ^n}^{d-1}
	\suchthat
	T_i
	\leq
	\supnorm{\vecsuper{x}{i}}
	\leq
	2T_i
	\quad\brax{  i = 1,\dotsc, d-1 }
	\\
	\supnorm{\gradFMultilinear{\vec{\beta}}{\vecsuper{x}{1},\dotsc,\vecsuper{x}{d-1}}}
	\leq
	\supnorm{\vec{\beta}}
	B^{d-2}
	\Big\},
	\end{multline*}
	so that
	\begin{equation}\label{3.eqn:aux_ineq_in_dyadic_blocks}
	\auxIneqNumSolns(B)
	\leq
	1+
	\sum_{\substack{t_1,\dotsc,t_{d-1}\in \bbN \\ t_i < \log_2 B}}
	\#Z(2^{t_1},\dotsc,2^{t_{d-1}}).
	\end{equation}
	Let $C_1$ be a positive real number which is  sufficiently large in terms of $\vec{f}$. The trivial bound $\#K(T_1,\dotsc,T_{d-1}) \ll_n T_1^n\dotsm T_{d-1}^n$ gives
	\[
	\sum_{\substack{t_1,\dotsc,t_{d-1}\in \bbN \\ t_1+\dotsb+t_{d-1} <  \log_2 C_1 B^{d-2} }}
	\#Z(2^{t_1},\dotsc,2^{t_{d-1}})
	\ll_{d,n}
	B^{(d-2)n} \brax{\log C_1 B}^{d-1},
	\]
	and substituting this into \eqref{3.eqn:aux_ineq_in_dyadic_blocks} gives
	\begin{multline}
	\auxIneqNumSolns(B)
	\ll_{d,n,C_1}
	B^{(d-2)n} \brax{\log B}^{d-1}
	\\
	+
	\sum_{\substack{t_1,\dotsc,t_{d-1}\in \bbN \\ t_i < \log_2 B \\ t_1+\dotsb+t_{d-1} \geq \log_2 C_1 B^{d-2} }}
	\#Z(2^{t_1},\dotsc,2^{t_{d-1}}).
	\label{3.eqn:aux_ineq_in_nontrivial_dyadic_blocks}
	\end{multline}
	For the remainder of the proof, we will let $T_1,\dotsc,T_{d-1}\in\open{0}{B}$ such that
	\begin{equation}\label{3.eqn:nontrivial_dyadic_range}
	T_1\dotsm T_{d-1}
	\geq
	C_1
	B^{d-2},
	\end{equation}
	and we will prove that
	\begin{equation}\label{3.eqn:aux_ineq_in_one_nontrivial_dyadic_block}
	\#Z(T_1,\dotsc,T_{d-1})
	\ll_{\vec{f}}
	B^{(d-2)n+R-1} \brax[\Big]{\frac{T_1\dotsm T_{d-1}}{B^{d-1}}}^{R-1}
	.
	\end{equation}
	Substituting \eqref{3.eqn:aux_ineq_in_one_nontrivial_dyadic_block} into \eqref{3.eqn:aux_ineq_in_nontrivial_dyadic_blocks} will then prove the proposition.
	
	We claim that for each $\brax{\vecsuper{x}{1},\dotsc,\vecsuper{x}{d-1}} \in Z(T_1,\dotsc,T_{d-1})$, there exist linear subspaces $U_1,\dotsc, U_{d-1}$ of $\bbR^n$,  spanned by standard basis vectors of $n$-space, such that $\dim U_1+ \dotsb+\dim U_{d-1}=n-\worstTangentSpaceDimLeadingPart$ and
	\begin{equation}\label{3.eqn:Jacobian_of_grad_f_with_dyadic_variables}
	\supnormbigg{ \betaFMultilinJacobian{\vecsuper{x}{1},\dotsc,\vecsuper{x}{d-1}}  
		\begin{smallpmatrix}\vecsuper{u}{1} \bigstrut \\ \hline \raisebox{5pt}{\scalebox{.75}{\vdots}} \\ \hline \vecsuper{u}{d-1} \end{smallpmatrix}
	} 
	\gg_{\vec{f}}
	\supnorm{\vec{\beta}} T_1 \dotsm T_{d-1}
	\max_{i=1,\dotsc,d-1}\frac{\supnorm{\vecsuper{u}{i}}}{T_i}
	\end{equation}
	for  all $\vecsuper{u}{i} \in U_i$. Indeed, if $\brax{\vecsuper{x}{1},\dotsc,\vecsuper{x}{d-1}} \in Z(T_1,\dotsc,T_{d-1})$, then
	\begin{align*}
	\supnorm{\gradFMultilinear{\vec{\beta}}{\vecsuper{x}{1},\dotsc,\vecsuper{x}{d-1}}}
	&\leq
	\supnorm{\vec{\beta}}B^{d-2},
	\intertext{and by \eqref{3.eqn:nontrivial_dyadic_range} it follows that}
	\supnorm{\gradFMultilinear{\vec{\beta}}{\vecsuper{x}{1},\dotsc,\vecsuper{x}{d-1}}}
	&\leq
	C_1^{-1} \supnorm{\vec{\beta}}
	T_1\dotsm T_{d-1}.
	\intertext{In particular,}
	\supnorm{\gradFMultilinear{\vec{\beta}}{\vecsuper{x}{1},\dotsc,\vecsuper{x}{d-1}}}
	&\leq
	C_1^{-1} \supnorm{\vec{\beta}}
	\supnorm{\vecsuper{x}{1}}\dotsm\supnorm{\vecsuper{x}{d-1}},
	\end{align*}
	and since we took $C_1\gg_{\vec{f}}1$ sufficiently large, we can apply Lemma~\ref{3.lem:Jacobian_of_grad_f} to give us spaces $U_i$ satisfying the required conditions.
	
	Fix some particular $U_i$, and fix  integral $n$-vectors  $\vecsuper{v}{1},\dotsc,\vecsuper{v}{d-1}$ satisfying  $T_i \leq \supnorm{\vecsuper{v}{i}}\leq 2T_i$ such that every $\vecsuper{v}{i}$ lies in the orthogonal complement of $U_i$. We then define $Z^{\ast}(T_1,\dotsc,T_{d-1})$ to be the subset of $Z(T_1,\dotsc,T_{d-1})$ containing those $(d-1)$-tuples $\brax{\vecsuper{x}{1},\dotsc,\vecsuper{x}{d-1}}$ which satisfy the bound \eqref{3.eqn:Jacobian_of_grad_f_with_dyadic_variables} for all $\vecsuper{u}{i} \in U_i$, and for which $\vecsuper{x}{i}-\vecsuper{v}{i}\in U_i$ for each $i$. We claim that
	\begin{equation}\label{3.eqn:aux_ineq_in_careful_region}
	\#
	Z^{\ast}(T_1,\dotsc,T_{d-1})
	\ll_{\vec{f}}
	\brax[\bigg]{\frac{B^{d-2}}{T_1\dotsm T_{d-1}}}^{n-\worstTangentSpaceDimLeadingPart}
	T_1^{\dim U_1}\dotsm T_{d-1}^{\dim U_{d-1}}.
	\end{equation}
	Every point in the set $Z(T_1,\dotsc,T_{d-1})$ lies in $Z^{\ast}(T_1,\dotsc,T_{d-1})$ for some choice of the parameters $U_i$ and $\vecsuper{v}{i}$. There are $O_{d,n}(1)$ choices for the spaces $U_i$, and for each one of these choices there are $O_{d,n}\brax{
		T_1^{n-\dim U_1}\dotsm T_{d-1}^{n-\dim U_{d-1}}}$  possibilities for the vectors $\vecsuper{v}{i}$, so by summing over all the possibilities we see that \eqref{3.eqn:aux_ineq_in_careful_region} implies
	\[
	\#Z(T_1,\dotsc,T_{d-1})
	\ll_{\vec{f}}
	\brax[\bigg]{\frac{B^{(d-2)}}{T_1\dotsm T_{d-1}}}^{n-\worstTangentSpaceDimLeadingPart}
	T_1^n\dotsm T_{d-1}^n
	\]
	which is the desired conclusion \eqref{3.eqn:aux_ineq_in_one_nontrivial_dyadic_block}.
	
	Let $\brax{\vecsuper{x}{1},\dotsc,\vecsuper{x}{d-1}}, \brax{\vecsuper{y}{1},\dotsc,\vecsuper{y}{d-1}} \in Z^{\ast}(T_1,\dotsc,T_{d-1})$ and for each $i$ write $\vecsuper{u}{i} = \vecsuper{y}{i}-\vecsuper{x}{i}$, so that $\vecsuper{u}{i}$ is an integral vector lying in $U_i$. We suppose that
	\begin{align}
	\supnorm{\vecsuper{u}{i}} &\leq C_1^{-1} T_i
	&\text{for all }i &= 1,\dotsc,d-1,
	\label{3.eqn:aux_ineq_solns_nearby}
	\intertext{
		where  $C_1$ is the sufficiently large constant from our assumption \eqref{3.eqn:nontrivial_dyadic_range}, and we will show that}
	\supnorm{\vecsuper{u}{i}} &\ll_{\vec{f}} \frac{B^{d-2}T_i}{T_1\dotsm T_{d-1}}
	&\text{for all }i &= 1,\dotsc,d-1.
	\label{3.eqn:repulsion_for_aux_ineq_solns}
	\end{align}
	From this it will follow that any box of the form
	\begin{multline*}
	A(\vec{\zeta})
	=
	\Big\{\brax{\vecsuper{\xi}{1},\dotsc,\vecsuper{\xi}{d-1}} \in (\bbR^n)^{d-1}
	\suchthat
	\text{for each }i = 1,\dotsc,d-1\text{ there are}
	\\
	\vecsuper{\nu}{i}\in U_i
	\text{ such that }
	\vecsuper{\xi}{i}
	=
	\vecsuper{\zeta}{i}+\vecsuper{\nu}{i} \text{ and }
	\supnorm{\vecsuper{\nu}{i}}\leq  C_1^{-1}T_i\Big\}
	\end{multline*}
	will satisfy 
	\[
	\#
	\setbrax[\big]{
		A(\vec{\zeta})\cup
		Z^{\ast}(T_1,\dotsc,T_{d-1})
	}
	\ll_{\vec{f}}
	\brax[\bigg]{\frac{B^{(d-2)}}{T_1\dotsm T_{d-1}}}^{n-R+1}
	T_1^{\dim U_1}\dotsm T_{d-1}^{\dim U_{d-1}}.
	\]
	We need at most $O_{\vec{f}}(1)$ such boxes to cover all of $Z(T_1,\dotsc,T_{d-1})$, so this implies our claim \eqref{3.eqn:aux_ineq_in_careful_region}.
	
	It remains to prove \eqref{3.eqn:repulsion_for_aux_ineq_solns}.   We have
	\[
	\gradFMultilinear{\vec{\beta}}{\vecsuper{y}{1},\dotsc,\vecsuper{y}{d-1}}
	=
	\gradFMultilinear{\vec{\beta}}{\vecsuper{x}{1}+\vecsuper{u}{1},\dotsc,\vecsuper{x}{d-1}+\vecsuper{u}{d-1}},
	\]
	and we will expand the right-hand side as a sum of terms of the type
	\begin{gather*}
	\gradFMultilinear{\vec{\beta}}{\vecsuper{u}{1},\vecsuper{x}{2},\dotsc,\vecsuper{x}{d-1}},\\
	\gradFMultilinear{\vec{\beta}}{\vecsuper{u}{1},\vecsuper{u}{2},\vecsuper{x}{3},\dotsc,\vecsuper{x}{d-1}},
	\end{gather*}
	and so on. That is, each term is equal to the system $\vecsuper{m}{\vec{\beta}\cdot\vec{f}}$ evaluated at a $(d-1)$-tuple of vectors, where we may take either $\vecsuper{x}{i}$ or $\vecsuper{u}{i}$ for the $i$th vector in the $(d-1)$-tuple. After grouping the terms together according to the number of vectors $\vecsuper{u}{i}$ occurring in the argument of $\vecsuper{m}{\vec{\beta}\cdot\vec{f}}$, this gives
	\begin{align}
	\MoveEqLeft[3]
	\gradFMultilinear{\vec{\beta}}{\vecsuper{y}{1},\dotsc,\vecsuper{y}{d-1}}
	\nonumber
	\\
	={}&
	\gradFMultilinear{\vec{\beta}}{\vecsuper{x}{1},\dotsc,\vecsuper{x}{d-1}}
	\nonumber
	\\
	&+
	\betaFMultilinJacobian{\vecsuper{x}{1},\dotsc,\vecsuper{x}{d-1}}  
	\begin{smallpmatrix}\vecsuper{u}{1} \bigstrut \\ \hline \raisebox{5pt}{\scalebox{.75}{\vdots}} \\ \hline \vecsuper{u}{d-1} \end{smallpmatrix}
	\nonumber
	\\
	&+
	O_{\vec{f}}\brax*{
		\supnorm{\vec{\beta}}
		\sum_{1\leq i_1<i_2 \leq d-1} \supnorm{\vecsuper{u}{i_1}} \cdot \supnorm{\vecsuper{u}{i_1}} \prod_{k\neq i_1,i_2} T_k}+\dotsb
	\nonumber
\\
&
	+
	O_{\vec{f}}\brax*{ 
		\supnorm{\vec{\beta}}\sum_{1\leq i_1< \dotsb < i_{d-2} \leq d-1} \supnorm{\vecsuper{u}{i_1}} \dotsm \supnorm{\vecsuper{u}{i_{d-2}}} \prod_{k\neq i_1,\dotsc,i_{d-2}} T_k}.
	\label{3.eqn:taylor_expanding_grad_f_I} 
	\end{align}
	By \eqref{3.eqn:aux_ineq_solns_nearby}, the total error in \eqref{3.eqn:taylor_expanding_grad_f_I} is
	\begin{equation*}
	O_{\vec{f}}\brax[\Big]{
		C_1^{-1}
		\supnorm{\vec{\beta}} T_1 \dotsm T_{d-1}
		\max_{i=1,\dotsc,d-1}\frac{\supnorm{\vecsuper{u}{i}}}{T_i}
	}.
	\end{equation*}
	In addition, as $\brax{\vecsuper{x}{1},\dotsc,\vecsuper{x}{d-1}}$ and $\brax{\vecsuper{y}{1},\dotsc,\vecsuper{y}{d-1}}$ belong to $Z(T_1,\dotsc,T_{d-1})$ we have
	\begin{equation*}
	\supnorm{\gradFMultilinear{\vec{\beta}}{\vecsuper{x}{1},\dotsc,\vecsuper{x}{d-1}}}, \,
	\supnorm{\gradFMultilinear{\vec{\beta}}{\vecsuper{y}{1},\dotsc,\vecsuper{y}{d-1}}}
	\leq \supnorm{\vec{\beta}} B^{d-2}.
	\end{equation*}
	So by \eqref{3.eqn:taylor_expanding_grad_f_I},
	\begin{multline*}
	\betaFMultilinJacobian{\vecsuper{x}{1},\dotsc,\vecsuper{x}{d-1}}  
	\begin{smallpmatrix}\vecsuper{u}{1} \bigstrut \\ \hline \raisebox{5pt}{\scalebox{.75}{\vdots}} \\ \hline \vecsuper{u}{d-1} \end{smallpmatrix}
	\\
	\ll_{\vec{f}}
	\supnorm{\vec{\beta}} B^{d-2}+
	C_1^{-1}
	\supnorm{\vec{\beta}} T_1 \dotsm T_{d-1}
	\max_{i=1,\dotsc,d-1}\frac{\supnorm{\vecsuper{u}{i}}}{T_i}.
	\end{multline*}
	By \eqref{3.eqn:Jacobian_of_grad_f_with_dyadic_variables}  this implies that
	\begin{multline*}
	\supnorm{\vec{\beta}} T_1 \dotsm T_{d-1}
	\max_{i=1,\dotsc,d-1}\frac{\supnorm{\vecsuper{u}{i}}}{T_i}
	\\
	\ll_{\vec{f}}
	\supnorm{\vec{\beta}} B^{d-2}+
	C_1^{-1}
	\supnorm{\vec{\beta}} T_1 \dotsm T_{d-1}
	\max_{i=1,\dotsc,d-1}\frac{\supnorm{\vecsuper{u}{i}}}{T_i}.
	\end{multline*}
	At the start of the proof we assumed that $\vec{\beta}\neq \vec{0}$ and that $C_1\gg_{\vec{f}}1 $, so this implies the conclusion \eqref{3.eqn:repulsion_for_aux_ineq_solns}.
\end{proof}

\section{Proof of Proposition~\ref{3.prop:general_position_condition}}\label{3.sec:general_position}

In this section we will prove Proposition~\ref{3.prop:general_position_condition}, bounding the quantity $\worstTangentSpaceDimCapitalH$ for typical systems $\vec{H}$. The strategy is to relate $\sigma^\ast$ to the dimension of a certain explicit complex variety $W$, which we will be able to parametrise.

\begin{proof}[Proof of Proposition~\ref{3.prop:general_position_condition}]
	If $\vec{H}$ is a system of degree $d$ forms with coefficients in a field $\bbF$, we define a subvariety $\Sigma_{\vec{H}}$ of $\bbP_{\bbF}^{R-1}\times (\bbP_{\bbF}^{n-1})^{d-1}$ as follows. Taking $\vec{\beta}$ and the $\vecsuper{x}{i}$ as homogeneous coordinates on $\bbP_{\bbF}^{R-1}$ and $\bbP_{\bbF}^{n-1}$ respectively, $\Sigma_{\vec{H}}$ is cut out by the conditions
	\begin{align*}
	\gradSomethingMultilinear{\vec{\beta}\cdot\vec{H}}{\vecsuper{x}{1},\dotsc,\vecsuper{x}{d-1}}&= \vec{0},\\
	\rank \betaDotHMultilinJacobian{\vecsuper{x}{1},\dotsc,\vecsuper{x}{d-1}} &\leq n-R.
	\end{align*}
	The condition that $\Sigma_{\vec{H}}$ be nonempty cuts out a Zariski closed subset, defined over \tbbQ, in the space of all systems $\vec{H}$. We will show that this is a proper subset. This will prove the proposition, because by \eqref{3.eqn:def_of_sigma-star} we have $\worstTangentSpaceDimCapitalH\geq R$ precisely when the variety $\Sigma_{\vec{H}}$ has an$\bar{\bbF}$-point.
	
	Suppose for a contradiction that $\Sigma_{\vec{H}}$ is nonempty for every system $\vec{H}$.
	
	Let $N(d,n)$ be the number of coefficients of a general form of degree $d$ in $n$ variables. The space $\bbP_\bbQ^{N(d,n)-1}$ parametrises degree $d$ forms in $n$ variables up to multiplication by a constant. Let $\Sigma_0$ be the subvariety of $\bbP_\bbQ^{N(d,n)-1}\times (\bbP_\bbQ^{n-1})^{d-1}$ defined by the two conditions
	\begin{equation}
	\label{3.eqn:bad_forms_I}
	\begin{aligned}
	\gradSomethingMultilinear{H}{\vecsuper{x}{1},\dotsc,\vecsuper{x}{d-1}}&= \vec{0},\\
	\rank \HMultilinJacobian{\vecsuper{x}{1},\dotsc,\vecsuper{x}{d-1}} &\leq n-R,
	\end{aligned}
	\end{equation}
	where  the form $H$ represents a point of $\bbP_\bbC^{N(d,n)-1}$. Given a system $\vec{H}$ with linearly independent $H_i$, we can embed the variety $\Sigma_{\vec{H}}$ into $\Sigma_0$ by sending the vector $\vec{\beta}$ to the form $\vec{\beta}\cdot\vec{H}$. The image of this embedding is $\Theta\cap\Sigma_0$, where $\Theta$ is the projective linear subspace of $\bbP_\bbQ^{N(d,n)-1}$ spanned by the  $H_i$. Since every variety  $\Sigma_{\vec{H}}$ is nonempty by assumption, the intersection $\Theta\cap\Sigma_0$ is nonempty for every $(R-1)$-dimensional projective linear space  $\Theta$ in $\bbP_\bbQ^{N(d,n)-1}$. So we have
	\begin{equation}\label{3.eqn:bad_forms_I_dim}
	\dim \Sigma_0
	\geq
	N(d,n)-R.
	\end{equation}
	Now let $\Sigma_1$ be the subvariety of $\bbP_\bbQ^{N(d,n)-1}\times(\bbP_\bbQ^{n-1})^d$ cut out by the conditions
	\begin{equation}\label{3.eqn:bad_forms_II}
	\begin{aligned}
	\gradSomethingMultilinear{H}{\vecsuper{x}{1},\dotsc,\vecsuper{x}{d-1}}&= \vec{0},\\
	\HMultilinJacobian{\vecsuper{x}{1},\dotsc,\vecsuper{x}{d-1}}^T \vecsuper{x}{d} &= \vec{0},
	\end{aligned}
	\end{equation}
	where $H$ represents a point of $\bbP_\bbQ^{N(d,n)-1}$ and $\vecsuper{x}{1},\dotsc,\vecsuper{x}{d}$ are vectors of homogeneous coordinates on $\bbP_\bbQ^{n-1}$. Each solution of \eqref{3.eqn:bad_forms_I} corresponds to an $R$-dimensional space of vectors $\vecsuper{x}{d}$ satisfying \eqref{3.eqn:bad_forms_II}. So each point of $\Sigma_0$ corresponds to an $(R-1)$-dimensional projective space of points on $\Sigma_1$, and by \eqref{3.eqn:bad_forms_I_dim} we must have
	\begin{equation}\label{3.eqn:lots_of_singular_points}
	\dim \Sigma_1
	\geq
	N(d,n)-1.
	\end{equation}	
	Consider the map
	\begin{equation}\label{3.eqn:def_of_linear_map_on_f}
	H
	\mapsto
	\left(
	\begin{array}{@{}c@{}}
	\HMultilinJacobian{\vecsuper{x}{1},\dotsc,\vecsuper{x}{d-1}}^T \vecsuper{x}{d}
	\bigstrut
	\\
	\hline
	\bigstrut
	\gradSomethingMultilinear{H}{\vecsuper{x}{1},\dotsc,\vecsuper{x}{d-1}}
	\end{array}
	\right),
	\end{equation}
	where the right-hand side is a vector with $2n$ entries obtained by concatenating two vectors with $n$ entries each. This map is linear in the coefficients of $H$. Let $L\brax{\vecsuper{x}{1},\dotsc,\vecsuper{x}{d}}$ be the matrix of this linear map, so that $L\brax{\vecsuper{x}{1},\dotsc,\vecsuper{x}{d}}$ is a $(dn)\times N(d,n)$ matrix whose entries are polynomials in the $\vecsuper{x}{i}$ with rational coefficients. Given a $d$-tuple  $\brax{\vecsuper{x}{1},\dotsc,\vecsuper{x}{d}}$, the space of forms $H$ satisfying \eqref{3.eqn:bad_forms_II} has dimension equal to
	\[
	N(d,n)-\rank L\brax{\vecsuper{x}{1},\dotsc,\vecsuper{x}{d}}.
	\]
	So if we  let $\Lambda(k)$ be the subvariety of $(\bbP_\bbQ^{n-1})^d$ cut out by the condition
	\begin{equation*}
	\rank L\brax{\vecsuper{x}{1},\dotsc,\vecsuper{x}{d}} = k,
	\end{equation*}
	then each point on $\Lambda(k)$ corresponds to a $(N(d,n)-k-1)$-dimensional projective linear space on $\Sigma_1$, and hence
	\begin{equation*}
	\dim \Sigma_1 = \max_{k\in \set{0,\dotsc,dn}}  \dim \Lambda(k)+(N(d,n)-k-1).
	\end{equation*}
	In particular, \eqref{3.eqn:lots_of_singular_points} implies that for some $k_0\in \set{0,\dotsc,dn}$ we have
	\begin{equation}
	\label{3.eqn:fibering_bad_forms_II}
	\dim \Lambda(k_0)\geq k_0.
	\end{equation}
	Let $W$ be the variety cut out in $(\bbP_\bbQ^{n-1})^d\times\bbP_\bbQ^{dn-1}$ by the equation
	\begin{equation}
	\vec{w}^T L\brax{\vecsuper{x}{1},\dotsc,\vecsuper{x}{d}}  = \vec{0}
	\label{3.eqn:def_of_variety_W},
	\end{equation}
	where $\vec{w}$ is a vector of homogeneous coordinates on $\bbP_\bbQ^{dn-1}$. If we are given a $d$-tuple $\brax{\vecsuper{x}{1},\dotsc,\vecsuper{x}{d}}$ representing a point on $\Lambda(k_0)$, then the space of vectors $\vec{w}$ satisfying \eqref{3.eqn:def_of_variety_W} has dimension
	\[
	dn-k_0.
	\]
	So each point on $\Lambda(k_0)$ corresponds to a $(dn-k_0-1)$-dimensional projective linear space on $W$, and \eqref{3.eqn:fibering_bad_forms_II} implies that
	\begin{equation}
	\label{3.eqn:points_which_are_singular_for_all_f}
	\dim W
	\geq
	dn-1.
	\end{equation}
	We will show that the complex points $W(\bbC)$ can be parametrised by $dn-2$ complex parameters. By standard results this implies that $
	\dim W
	\leq dn-2$, see the remarks at the end of \S2.3 in Chapter~2 of Shafarevich~\cite{shafarevichBasicAlgGeomI}.
	
	Let $ \brax{ \vecsuper{x}{1},\dotsc,\vecsuper{x}{d}, \vec{w} }$ represent a $\bbC$-point of $W$. From the definitions \eqref{3.eqn:def_of_linear_map_on_f} and \eqref{3.eqn:def_of_variety_W} we see that the expression
	\begin{align*}
	\vec{w}^T
	\left(
	\begin{array}{@{}c@{}}
	\HMultilinJacobian{\vecsuper{x}{1},\dotsc,\vecsuper{x}{d-1}}^T \vecsuper{x}{d}
	\bigstrut
	\\
	\hline
	\bigstrut
	\gradSomethingMultilinear{H}{\vecsuper{x}{1},\dotsc,\vecsuper{x}{d-1}}
	\end{array}
	\right)
	={}&
	(\vecsuper{x}{d})^T \HMultilinJacobian{\vecsuper{x}{1},\dotsc,\vecsuper{x}{d-1}}
	\begin{smallpmatrix}w_{1}  \\  \raisebox{4pt}{\scalebox{.75}{\vdots}} \\  w_{(d-1)n} \end{smallpmatrix}
	\\&+
	\gradSomethingMultilinear{H}{\vecsuper{x}{1},\dotsc,\vecsuper{x}{d-1}}^T
	\begin{smallpmatrix}w_{(d-1)n+1 } \\  \raisebox{4pt}{\scalebox{.75}{\vdots}} \\ w_{dn} \end{smallpmatrix}
	\end{align*}
	must vanish uniformly for all degree $d$ forms $H$. In the special case when $H(\vec{x}) = \brax{\vec{b}\cdot\vec{x}}^d$, so that the form $H$ is a $d$th power of a linear form, we calculate from the definition of $\HMultilinJacobian{\vecsuper{x}{1},\dotsc,\vecsuper{x}{d-1}}$ in \eqref{3.eqn:def_of_multilin_Jacobian} 	that this expression is
	\begin{align*}
	\brax{\vec{b}\cdot\vecsuper{w}{1}}
	\brax{\vec{b}\cdot\vecsuper{x}{2}}\dotsm 
	\brax{\vec{b}\cdot\vecsuper{x}{d}}
	&+
	\brax{\vec{b}\cdot\vecsuper{x}{1}}
	\brax{\vec{b}\cdot\vecsuper{w}{2}}
	\brax{\vec{b}\cdot\vecsuper{x}{3}}
	\dotsm 
	\brax{\vec{b}\cdot\vecsuper{x}{d}}
	\\
	&+
	\brax{\vec{b}\cdot\vecsuper{x}{1}}\dotsm 
	\brax{\vec{b}\cdot\vecsuper{x}{d-1}}
	\brax{\vec{b}\cdot\vecsuper{w}{d}},
	\end{align*}
	where we split $\vec{w}$ into $d$ separate $n$-vectors $\vecsuper{w}{i}$, given by
	\begin{equation}\label{3.eqn:def_of_w^i}
	\vec{w}
	=
	\begin{smallpmatrix}\vecsuper{w}{1} \bigstrut \\ \hline \raisebox{5pt}{\scalebox{.75}{\vdots}} \\ \hline \vecsuper{w}{d} \end{smallpmatrix}.
	\end{equation}
	We may divide through by $\brax{\vec{b}\cdot\vecsuper{x}{1}}
	\dotsm 
	\brax{\vec{b}\cdot\vecsuper{x}{d}}$ to see that
	\begin{equation}\label{3.eqn:thing_which_always_vanishes}
	\frac{\vec{b}\cdot\vecsuper{w}{1}}{\vec{b}\cdot\vecsuper{x}{1}}
	+\dotsb+
	\frac{\vec{b}\cdot\vecsuper{w}{d}}{\vec{b}\cdot\vecsuper{x}{d}}
	=0
	\end{equation}
	whenever $\vec{b}\in\bbC^n$ and all of the denominators $\vec{b}\cdot\vecsuper{x}{i} \neq 0$ are nonzero.

	Below we will find $m\in\set{1,\dotsc,d}$,  $\vec{k}\in \set{1,\dotsc,m}^d$,   $\vec{\lambda}\in\brax{\bbC\setminus\set{0}}^d$, $\vec{\mu}\in\bbC^m$ and $
	\vecsuper{y}{1},\dotsc,\vecsuper{y}{m}\in\bbC^n\setminus\set{\vec{0}}$ such that
	\begin{align}
	\sum_{k_j=\ell} \lambda_j &=1&(\ell=1,\dotsc,m),
	\label{3.eqn:restriction_on_coeffs_lambda}
	\\
	\mu_1+\dotsb+\mu_m &= 0,
	\label{3.eqn:restriction_on_coeffs_mu}
	\\
	\lambda_i\vecsuper{x}{i}
	&=  \vecsuper{y}{k_i}
	&(i=1,\dotsc,d),
	\label{3.eqn:restriction_on_x^i}
	\\
	\sum_{k_j=\ell} \lambda_i \vecsuper{w}{j} &= \mu_\ell\vecsuper{y}{\ell}
	&(\ell=1,\dotsc,m),
	\label{3.eqn:restriction_on_w^i}
	\end{align}
	where the $\vecsuper{w}{i}$ are as in \eqref{3.eqn:def_of_w^i}. Given $m$ and $ \vec{k}$, we have an $mn$-dimensional space of parameters $(\vecsuper{y}{1},\dotsc,\vecsuper{y}{m})$ and a $(d-1)$-dimensional space of parameters $(\vec{\lambda},\vec{\mu})$ satisfying \eqref{3.eqn:restriction_on_coeffs_lambda} and \eqref{3.eqn:restriction_on_coeffs_mu}. Having chosen the values of these parameters the value of each $\vecsuper{x}{i}$ is fixed uniquely by \eqref{3.eqn:restriction_on_x^i}, and there is a $(d-m)n$-dimensional space of vectors $\vec{w}$ satisfying \eqref{3.eqn:restriction_on_w^i}. In total then, the space of possible $(d+1)$-tuples $(\vecsuper{x}{1},\dotsc,\vecsuper{x}{d},\vec{w})$ has dimension at most
	\[
	mn+(d-1)+(d-m)n = dn+d-1.
	\]
	For any $u_1,\dotsc, u_d, v \in\bbC\setminus\set{\vec{0}}$, the $(d+1)$-tuple $ \brax{ u_1\vecsuper{x}{1},\dotsc,u_d\vecsuper{x}{d}, v\vec{w} }$ represents the same point of $W(\bbC)$ as $ \brax{ \vecsuper{x}{1},\dotsc,\vecsuper{x}{d}, \vec{w} }$. Consequently $ W(\bbC)$ can be parametrised with $dn-2$ complex parameters, and by the comments after \eqref{3.eqn:points_which_are_singular_for_all_f} this gives a contradiction and  proves the proposition.
	
	It remains to find, for each $(d+1)$-tuple $ \brax{ \vecsuper{x}{1},\dotsc,\vecsuper{x}{d}, \vec{w} }$ satisfying \eqref{3.eqn:def_of_variety_W}, a choice of the parameters $m, \vec{k}, \vec{\lambda}, \vec{mu}$ and $\vecsuper{y}{i}$ such that the relations \eqref{3.eqn:restriction_on_coeffs_lambda}--\eqref{3.eqn:restriction_on_w^i} hold. Define an equivalence relation on the set $\set{\vecsuper{x}{1},\dotsc,\vecsuper{x}{d}}$ by saying that  $\vecsuper{x}{i}$ and $\vecsuper{x}{j}$ are equivalent if they are linearly dependent. Let $m$ be the number of equivalence classes. Number them from 1 to $m$, and let $k_i$ be the number of the equivalence class to which $\vecsuper{x}{i}$ belongs. All the vectors in a given equivalence class are multiples of one fixed vector, so there are nonzero scalars $\lambda_1,\dotsc,\lambda_d$ and nonzero vectors $\vecsuper{y}{1},\dotsc,\vecsuper{y}{m}$  satisfying \eqref{3.eqn:restriction_on_x^i}. By replacing each $\vecsuper{y}{\ell}$ with a scalar multiple of itself if necessary, we may assume that \eqref{3.eqn:restriction_on_coeffs_lambda} holds. It remains to find $\vec{\mu}\in \bbC^m$ satisfying \eqref{3.eqn:restriction_on_coeffs_mu} and \eqref{3.eqn:restriction_on_w^i}.
	
	Substituting \eqref{3.eqn:restriction_on_x^i} into \eqref{3.eqn:thing_which_always_vanishes} shows that
	\begin{equation}\label{3.eqn:thing_which_always_vanishes_redux}
	\sum_{\ell=1}^m \frac{\vec{b}\cdot\sum_{k_i=\ell} \lambda_i \vecsuper{w}{i}}{\vec{b}\cdot\vecsuper{y}{\ell}} 
	=0
	\end{equation}
	whenever $\vec{b}\in\bbC^n$ and none of the denominators $\vec{b}\cdot\vecsuper{y}{\ell}$ vanish. Let $\ell_0 \in \set{1,\dotsc,m}$, let $\vec{t}\in \bbC^n$ and suppose that $\vec{t}\cdot\vecsuper{y}{\ell_0}=0$. Since the inequalities $m \leq d \leq n$ hold, there exist $\vec{u}, \vec{v}\in\bbC^n$ satisfying the three conditions
	\begin{align}\label{3.eqn:magic_choice_of_form_I}
	\vec{u}\cdot\vecsuper{y}{\ell}
	&
	\neq 0
	&\text{for all }\ell\neq \ell_0,
	\\
	\nonumber
	\vec{u}\cdot\vecsuper{y}{\ell_0}
	&
	= 0,
	&\text{ and}
	\\
	\label{3.eqn:magic_choice_of_form_III}
	\vec{v}\cdot\vecsuper{y}{\ell_0}
	&= 1.
	\end{align}
	For some small $\epsilon>0$ we set
	\[
	\vec{b}
	=
	\vec{t}+\epsilon \vec{u} + \epsilon^2 \vec{v}.
	\]
	Then the conditions \eqref{3.eqn:magic_choice_of_form_I} and \eqref{3.eqn:magic_choice_of_form_III} ensure that
	\begin{align*}
	\vec{b}\cdot\vecsuper{y}{\ell}
	&\gg
	\epsilon
	&\text{for all }\ell\neq \ell_0,\text{ and}
	\\
	\vec{b}\cdot\vecsuper{y}{\ell_0}
	&=
	\epsilon^2.
	\end{align*}
	So \eqref{3.eqn:thing_which_always_vanishes_redux} implies that
	\begin{equation*}
	\epsilon^{-2} \vec{t}\cdot \sum_{k_i=\ell_0}\lambda_i \vecsuper{w}{i}
	=O(\epsilon^{-1}).
	\end{equation*}
	Letting $\epsilon \to 0$ we see that $\vec{t}\cdot \sum_{k_i=\ell_0}\lambda_i \vecsuper{w}{i} = 0$ vanishes. Recall that this holds for any $\ell_0 \in \set{1,\dotsc,m}$ and any $\vec{t}\in \bbC^n$, provided only that $\vec{t}\cdot\vecsuper{y}{\ell_0}$ vanishes. So for each $\ell_0 \in \set{1,\dotsc,m}$ there must be some $\mu_{\ell_0}\in\bbC$ such that
	\[
	\sum_{k_i=\ell_0}\lambda_i \vecsuper{w}{i}
	=\mu_{\ell_0}\vecsuper{y}{\ell_0}.
	\]
	This gives us an $m$-vector $\vec{\mu}$ satisfying \eqref{3.eqn:restriction_on_w^i}. Finally, substituting \eqref{3.eqn:restriction_on_w^i} into \eqref{3.eqn:thing_which_always_vanishes_redux} shows that
	\[
	\sum_{\ell=1}^m \mu_\ell \frac{\vec{b}\cdot\vecsuper{y}{\ell}}{\vec{b}\cdot\vecsuper{y}{\ell}} = \mu_1+\dotsb+\mu_m =0,
	\]
	which proves \eqref{3.eqn:restriction_on_coeffs_mu}. So \eqref{3.eqn:restriction_on_coeffs_lambda}--\eqref{3.eqn:restriction_on_w^i} all hold, as required.
\end{proof}

\subsection*{Funding}
This work was supported by the Engineering and Physical Sciences Research Council [EP/J500495/1, EP/M507970/1]; and by the Fields Institute for Research in Mathematical Sciences.

\subsection*{Acknowledgements}
This paper is based on a DPhil thesis submitted to Oxford University. I would like to thank my DPhil supervisor, Roger Heath-Brown. I would like to thank Winston Heap for helpful comments.

\bibliography{systems-of-many-forms}

\end{document}